\documentclass[11pt]{article}
\usepackage{amsmath}

\usepackage{amsfonts}
\usepackage{amssymb}
\usepackage{euscript}
\newenvironment{proof}{\noindent {\it Proof.~~}\ }{\  \rule{1mm}{2mm}\medskip}

\newenvironment{proof*}{\noindent {\it Proof.~~}\ }{}

\newtheorem{theorem}{Theorem}
\newtheorem{lemma}[theorem]{Lemma}
\newtheorem{corollary}[theorem]{Corollary}
\newtheorem{proposition}[theorem]{Proposition}

\newtheorem{theirtheorem}{Theorem}

\newtheorem{theirlemma}[theirtheorem]{Lemma}

\def\G{\partial}

\begin{document}
\title{Extensions of the  Scherck-Kemperman Theorem }
\author{ {Y. O. Hamidoune}\thanks{
UPMC Univ Paris 06,
 E. Combinatoire, Case 189, 4 Place Jussieu,
75005 Paris, France.}
}

\maketitle

%11B60, 11B34, 20D60.

\begin{abstract}
Let $\Gamma =(V,E)$  be a  reflexive relation
with a transitive automorphisms group. Let $v\in V$ and let $F$ be a finite subset of $V$ with $v\in F.$

We prove that the size of $\Gamma  (F)$ (the image of $F$ ) is at least  $$ |F|+ |\Gamma  (v)|-|\Gamma ^- (v)\cap F|.$$

Let  $A,B$ be finite subsets of a group $G.$
Applied to Cayley graphs, our result reduces to  following extension of the Scherk-Kemperman Theorem, proved by Kemperman:
$$|AB|\ge |A|+|B|-|A\cap (cB^{-1})|,$$
 for every  $c\in AB.$
 \end{abstract}

\section{Introduction}

Let $A,B$ be finite subsets of a group $G.$
The {\em Minkowsky product} of $A$ by $B$ is defined as  $AB=\{xy: x\in A \ \mbox{and}\ y\in B\}.$

Solving a problem of Moser,  Scherk  proved in \cite{sch} that $$|AB|\ge |A|+|B|-1,$$ if $G$ is abelian and if $A\cap B^{-1}=\{1\}.$  The validity of this result in the non-abelian case was proved by Kemperman \cite{kempcompl}. Kemperman mentioned in \cite{kempcompl} that this generalization  was independently obtained by Wehn. This result is known as the
Scherk-Kemperman Theorem.

Kemperman's Theorem  \cite{kempcompl} states that
 $$|AB|\ge |A|+|B|-|A\cap (cB^{-1})|,$$ for every  $c\in AB.$

The reader may find several applications of the Scherk-Kemperman Theorem  to the Theory of Non-unique factorization in the text book of Geroldinger-Halter-Koch \cite{gerlodinger}. Recall that this result  is used, among other tools, by Olson in
\cite{olsonaa},  to prove that for any  subset $S$ of a finite group $G$ with $|S|\ge 3\sqrt{|G|},$ there exist distinct elements $x_1, \cdots, x_k\in S$ with  $x_1 \cdots x_k=1.$
The Scherk-Kemperman Theorem is a basic tool in the proof by Gao
that a sequence of elements of a finite abelian  group $G$ with length
$|G|+d(G)$ contains a $|G|$--sub-sequence summing to $0,$ where $d(G)$ is the maximal size of a sequence of elements of $G$ having no non-empty zero-sum subsequence \cite{gaotnd}. A recent generalization of Gao's Theorem, based also on the Scherk-Kemperman Theorem, is contained in \cite{HweightD}.
Let $G$ be a group and let $B$ be a finite subset with $1\notin B.$
Using the Scherk-Kemperman Theorem, Eliahou and  Lecouvey  proved in \cite{el} that
 there is a permutation $\sigma$ of $B$ such that $x\sigma (x)\notin B,$ for every
$x\in B.$

By a {\em graph} we shall mean a relation.
Let $\Gamma =(V,E)$ be a  graph and let $F$ be a subset of $V.$
As a consequence of this definition of a graph, $\Gamma (F)$ is just the image of $F$ by the relation $\Gamma,$ defined in elementary Set Theory. The graph $\Gamma =(V,E)$ is said to be {\em locally-finite } if $\Gamma (x)$ is finite for all $x\in V.$

Suppose now that  $\Gamma $ is a loopless finite graph  with a transitive group of automorphisms. Let $v\in V$
and put $r=|\Gamma (v)|.$
Motivated by some conjectures from Graph Theory, Mader \cite{mader}, proved that there are   directed cycles  $C_1, \cdots , C_r$ such that $C_i\cap C_j=\{v\},$ for all $i<j.$

After reading the present work, the reader could certainly use Menger's Theorem to prove  that Mader's Theorem applied to
Cayley graphs, is equivalent to  the
Scherk-Kemperman Theorem restricted to finite groups.

Notice  that Mader's formulation fails in the infinite case, since infinite  graphs  with a transitive group of automorphisms could be acyclic.
As a main tool, Mader introduced the notion of a vertex-fragment \cite{mader}, mentioning some related difficulties.

Motivated by   Moser's problem  and by
 Mader's vertex-fragments, we shall define  vertex-molecules.
 Our approach works in the infinite case too and leads to easier proofs. In particular, our approach avoids a duality between positive and negative vertex-fragments, in the spirit of the one introduced in \cite{HATOM}, used extensively in the  arguments of Mader \cite{mader}.

Let $\Gamma =(V,E)$  be a vertex-transitive reflexive locally-finite
graph and  let $v\in V.$ Let $F$ be a subset of $V$ with $v\in F.$

Our main result states that  $$|\Gamma (F)|\ge |F|+ |\Gamma  (v)|-|\Gamma ^- (v)\cap F|.$$
Applied to Cayley graphs, this result implies Kemperman's Theorem and the Scherk-Kemperman Theorem.
We give also  a simple proof of Mader's Theorem.
Our result applies also to problems considered by Nathanson in \cite{Nat0} and by the author \cite{Hejc,Hspheres}.

\section{Some Terminology}

Let $\Gamma =(V,E)$ be a  graph.
The elements of $V$  will be called  {\em vertices}.
 %The {\em reverse   }  of $\Gamma $ is by definition
%$\Gamma ^{-}=(V,E^{-})$, where $E^{-}=\{(x,y) : \   (y,x) \in E\}.$

The {\em sub-graph}  induced on a subset $X$ is defined as
$\Gamma [X]=(X,(X\times X)\cap E).$ The {\em degree } of a vertex $x$ is defined as  $d_{\Gamma}(x)=|\Gamma (x)|.$
Our degree is called outdegree in some Graph Theory text books.
Recall that the graph $\Gamma$ is {\em locally-finite},  if  $\Gamma $  has only finite degrees.
The graph $\Gamma$ will be called {\em regular},  if  all the vertices have  the same degree.

The {\em boundary} of a subset $X$ is defined as
 $$\partial _{\Gamma}(X)= \Gamma (X)\setminus X .$$ We also write
 $$\nabla _{\Gamma}(X)= V\setminus (X\cup \Gamma (X) ).$$

When the context is clear, the reference to $\Gamma$ will be omitted.

For any subset  $F$  of  $V,$
$V=F\cup \nabla(F)\cup \partial (F)$ is a partition (with possibly empty parts).
  Since $\Gamma (F)\cap\nabla(F)=\emptyset,$  we have $F\cap \Gamma ^-(\nabla(F))=\emptyset.$ Hence
\begin{equation}\label{id0}
\partial^{-}(\nabla(F))\subset \partial (F).\end{equation}

%The last observation, used extensively in the isoperimetric method \cite{Hast}, contains a useful duality.
 The next lemma contains a useful sub-modular inequality:

\begin{theirlemma} \cite{Hast}\label{partialsub}{Let $\Gamma =(V,E)$ be a  locally-finite   reflexive
graph and let   $X,Y$     be finite  subsets of $V.$ Then

 \begin{equation}\label{submodularity}
|\partial  (X \cup Y)|+|\partial  (X \cap Y)|\le |\partial  (X )|+|\partial  ( Y)|.
\end{equation}}

\end{theirlemma}

\begin{proof} Observe that
\begin{eqnarray*}
|\Gamma (X\cup Y)|&=&|\Gamma (X)\cup \Gamma (Y)|\\
&=&|\Gamma (X)|+|\Gamma (Y)|-|\Gamma(X)\cap \Gamma (Y)|\\
&\le& |\Gamma (X)|+|\Gamma (Y)|-|\Gamma(X\cap Y)|
\end{eqnarray*}
By subtracting the equation $|X\cup Y|=|X|+|Y|-|X\cap Y|,$ we obtain (\ref{submodularity}).\end{proof}

 A map $f : V  \longrightarrow V$ will be called a {\em homomorphism }if
$\Gamma (f(x))=f(\Gamma (x)),$ for every $x\in V.$ A bijective homomorphism is called an {\em automorphism}.
The graph $\Gamma$ will be called {\em vertex-transitive}  if for all $x,y\in V,$ there is an automorphism $f$
such that $y=f(x).$
Clearly a vertex-transitive graph is regular.

Let $G$ be group and let $F$ and $S$ be subsets of $G.$
The {\em Cayley graph} defined on $G$ by $S,$
is defined as  $\mbox{Cay} (G,S)=(G,E),$ where  $ E=\{ (x,y) : x^{-1}y \
\in S \}.$  Notice that left translations are automorphisms of $\mbox{Cay} (G,S).$
In particular,
 Cayley  graphs are  vertex-transitive.

Putting $\Gamma =\mbox{Cay} (G,S),$   we have
clearly
 $\Gamma (F)=FS.$ Notice that Moser's condition "$A\cap B^{-1}=\{1\}$" is just "$A\cap \Gamma ^-(1)=\{1\}$".

\section{Moser's sets}
We investigate in this section a generalization of Moser's problem to graphs.

 Let $\Gamma =(V,E)$ be a locally-finite  reflexive graph and let $v\in V.$
A set $F$ is said to be a $v$--Moser's set if $\Gamma ^-(v)\cap F=\{v\}.$ Put
$$\mu (v,\Gamma)=\min \{|\partial (X)| : X \ \mbox{is a} \ v\mbox{--Moser's set of } \Gamma\}.$$

A Moser's set $X$ with $|\partial (X)|=\mu (v,\Gamma )$ will be called a $v$--{\em molecule}.
The reference to $\Gamma$ could be implicit.
%A Moser set of $\Gamma$ will be called a {\em negative} Moser set.

\begin{lemma}
 Let $\Gamma =(V,E)$ be a locally-finite  reflexive graph and let $v\in V.$ Then
the intersection
and the union of two $v$--molecules are $v$--molecules.

\end{lemma}

\begin{proof}

Notice that the intersection
and the union of two $v$--Moser sets are $v$--Moser sets.
Let $F_1$ and $F_2$ be two $v$--molecules.
Using   (\ref{submodularity}), we have

\begin{eqnarray*}
2\mu (v) &\le& |\partial  (F_1 \cap F_2)|+|\partial  (F_1 \cup F_2)|\\&\le& |\partial  (F_1 )|+|\partial ( F_2)|\le 2\mu (v).
\end{eqnarray*}
Thus, $F_1 \cap F_2$ and $F_1 \cup F_2$ are $v$--molecules.\end{proof}

In particular, there exists a $v$--molecule  contained in  every $v$--molecule.
Such a $v$--molecule will be called the $v$--{\em kernel} and denoted  by $K_v.$

In the finite case, one may prove that a Mader's $v$--fragment  \cite{mader}
is either a  $v$--molecule or a set obtained by deleting $v$ from a $v$--molecule.
Also the Mader's $v$--atom \cite{mader} is just $A_v=K_v\setminus \{v\}.$ However our approach
leads to simplifications, since we do not need negative $v$--fragments used by Mader. Notice that
negative $v$--fragments have no nice behavior in the infinite case.

\begin{lemma} \label{distinct}
 Let $\Gamma =(V,E)$ be a locally-finite  reflexive graph and let $v,w\in V$
 be  vertices. If $v\neq w,$ then
$K_v\neq K_w.$

Moreover    $\phi (K_v)=K_{\phi (v)},$  for any automorphism $\phi$ of $\Gamma.$

\end{lemma}

\begin{proof}
Assume that $v\neq w$  and that $K_v= K_w.$ It follows that $\Gamma ^-(w)\cap K_v=\{w\}.$ Thus $w\notin \G (K_v\setminus \{w\}).$  Using the definition of $\mu ,$ we have  $$\mu(v)\le |\G (K_v\setminus \{w\})|\le |\G (K_v)|=  \mu (v).$$ It would follow that $K_v\setminus \{w\}$ is a $v$--molecule, a contradiction.

Clearly $\phi( K_v)$ is a $\phi(v)$--Moser's set. In particular,
$K_{\phi (v)}\subset \phi( K_v).$ The reverse inclusion follows since $\phi ^{-1} (K_{\phi (v)})$ is a $v$--Moser's set.\end{proof}

The  next  Lemma generalizes to the infinite case a lemma of Mader \cite{mader}. Note that
Mader's argument is not suitable in the infinite case, since it involves negative $v$--fragments.

\begin{lemma} \label{davtrans}
Let $\Gamma =(V,E)$  be a  locally-finite reflexive
graph and  let $v,w\in V$ be vertices. If    $w\in K_v,$ then   either $v\in \Gamma (K_w)$
or $K_w\subset K_v.$
\end{lemma}

\begin{proof} Suppose that $w\in K_v$ and that $v\in \nabla (K_w).$ We have using (\ref{id0}),
  \begin{eqnarray*}\Gamma ^- (v)\subset \Gamma ^-(\nabla (K_w))&= &\nabla (K_w)\cup \partial ^-(\nabla (K_w))\\ &\subset& \nabla (K_w)\cup \partial (K_w)=V\setminus K_w.\end{eqnarray*}
  Therefore
$\Gamma ^- (v)\cap (K_v\cup K_w)=\Gamma ^- (v)\cap  K_v=\{v\}.$ Thus $K_v\cup K_w$ is a $v$--Moser's set.
Notice that $K_v\cap K_w$ is a $w$--Moser's set.

By (\ref{submodularity}), \begin{eqnarray*}
\mu (v)+\mu (w) &\le &|\partial  (K_v \cup K_w)|+|\partial  (K_v \cap K_w)|\\&\le& |\partial  (K_v )|+|\partial ( K_w)|\le \mu (v)+\mu (w).\end{eqnarray*}
Thus, $K_v \cap K_w$ is a  $w$--molecule and hence $K_w\subset K_v,$ a contradiction.\end{proof}

%We need the following easy property of $v$-kernels.

The {\em kernel--graph} $\Omega _{\Gamma}$, introduced in the finite case by Mader in \cite{mader}, is a graph on $V$ with
$$\Omega _{\Gamma}(v)=\Gamma (v)\cap K_v,$$ for every $v\in V.$ The reference to $\Gamma$ could be implicit.  The following easy lemma generalizes a result, proved by Mader \cite{mader} in  the finite case:

\begin{lemma} \label{moseraut}
Let $\Gamma =(V,E)$  be a  vertex-transitive reflexive locally-finite
graph.
Then
 the kernel-graph  $\Omega _{\Gamma}$ is a vertex-transitive graph.
Moreover  \begin{equation}\label{eqomgraph}\Omega ^- (v)\setminus \{v\}\subset \partial (K_v),\end{equation} for every $v\in K_v.$

\end{lemma}

\begin{proof} Take  an automorphism $\phi$ of $\Gamma.$ By 
Lemma \ref{distinct}, we have, $$\phi(\Omega (v))=\phi(\Gamma (v)\cap K_v)=\Gamma (\phi(v))\cap K_{\phi (v)}=\Omega (\phi(v)) ,$$
showing that
$\phi $ is an automorphism of $\Omega.$ Thus, $\Omega$ is vertex-transitive.

Take $w\in \ \Omega ^- (v),$ with $w\neq v.$
Take  an automorphism $\phi$ with $\phi(v)=w.$  We have by Lemma \ref{distinct}, $\phi (K_v)=K_{w},$
and hence $|K_v|=|K_w|.$
By Lemma \ref{distinct}, $K_{v}\not\subset  K_{w}.$
 By Lemma \ref{davtrans}, $v\in \Gamma (K_w).$ Since $v\in \Gamma (w),$ we have $w\notin K_v.$
\end{proof}

\section{A first step }

We prove here a special case of the main result. This special case implies the Scherk-Kemperman Theorem and Mader's Theorem.

\begin{theorem}\label{mainomega}
Let $\Gamma =(V,E)$  be a vertex-transitive reflexive locally-finite
graph and  let $v\in V.$ Then $\mu (v)\ge |\Gamma (v)|-|\Omega (v)|+|\Omega ^-(v)|-1.$

\end{theorem}

\begin{proof} Take an $x\in \Omega^-(v)\setminus \{v\}.$ We have $v\in K_x.$ By the definition of a Moser's set,
$x\notin \Gamma (v).$ Thus,
 $$\Omega^-(v)\cap \Gamma (v) =\{v\}.$$

 By (\ref{eqomgraph}), $$\Omega^-(v)\setminus \{v\}\subset  \partial (K_v).$$
Clearly $\Gamma (v)\setminus K_v \subset  \partial (K_v).$

It follows that $$\mu(v)=|\partial (K_v)|\ge |\Gamma(v)\setminus K_v|+|\Omega^-(v)|-1=|\Gamma (v)|-|\Omega (v)|+|\Omega ^-(v)|-1.$$
\end{proof}

Let $\Gamma =(V,E)$  be a vertex-transitive  locally-finite
graph and  let $v\in V.$ If $V$ is finite or if $\Gamma$ is a Cayley graph, then $|\Gamma (v)|=|\Gamma ^-(v)|.$
The last well known fact is an easy exercise. In order to deduce the Scherk-Kemperman Theorem, we
need  to show that the kernel graph of a Cayley graph is also a Cayley graph.

\begin{corollary}\label{mskw} (The  Scherk-Kemperman Theorem \cite{kempcompl})
Let  $A$ and $B$ be a finite subsets of a group $G$ with
$A\cap (B^{-1})=\{1\}.$
Then $|AB|\ge |A|+|B|-1.$
\end{corollary}
\begin{proof}
Put $\Gamma =\mbox{Cay} (G,B).$ Notice that $\Gamma$ is a reflexive vertex-transitive graph
and that $A$ is a $1$--Moser's set of $\Gamma.$

We
have  $\Omega (x)=K_x\cap \Gamma (x)=xK_1\cap xS=x(K_1\cap S),$  by Lemma \ref{distinct} and since a left translation is Cayley graph automorphism.
Thus  $$\Omega _{\Gamma}= \mbox{Cay} (G,B\cap K_1).$$ It follows that $$|AB|-|A|\ge
\mu (1)\ge |\Gamma (1)|-|\Omega (1)|+|\Omega ^-(1)|-1=|\Gamma (1)|-1=|B|-1.$$
\end{proof}

\begin{corollary}\label{finitemain}
Let $\Gamma =(V,E)$  be a  reflexive  finite
vertex-transitive graph and  let $v\in V.$ Let $F$ be a subset of $V$ with $v\in F.$  Then $$|\Gamma (F)|\ge |F|+ |\Gamma  (v)|-|\Gamma ^- (v)\cap F|.$$

\end{corollary}
\begin{proof}

Assume first that  $|\Gamma ^- (v)\cap F|=1.$
Clearly $F$
is a $v$--Moser's set.

By Lemma \ref{moseraut}, $\Omega _{\Gamma}$ is a vertex-transitive graph.
Since $V$ is finite, we have $|\Omega (v)|=|\Omega ^-(v)|.$
 By Theorem \ref{mainomega},
   $$|\Gamma (F)|- |F|\ge
\mu (v)\ge |\Gamma (v)|-|\Omega (v)|+|\Omega ^-(v)|-1=|\Gamma (v)|-1.$$

Assume now that  $|\Gamma ^- (v)\cap F|\ge 2.$ Put $F'=(F\setminus (\Gamma ^-(v)))\cup \{v\}.$
By the first case, $|\Gamma (F)|\ge|\Gamma (F')|\ge |F'|+|\Gamma (v)|-1=|F|+ |\Gamma  (v)|-|\Gamma ^- (v)\cap F|.$
\end{proof}

For the next result, we assume some familiarity with Menger's Theorem and with the notion of a directed  cycle in a graph.

\begin{corollary}(Mader \cite{mader})\label{maderth}
Let $\Gamma =(V,E)$  be a vertex-transitive loopless  finite
graph and  let $v\in V.$ Put $|\Gamma (v)|=r .$ Then there are   directed cycles  $C_1, \cdots , C_r$ such that $C_i\cap C_j=\{v\},$ for all $i<j.$
\end{corollary}

\begin{proof} Let $\Phi $ be the  graph obtained by a adding vertex $  v'\notin V,$ with $\Phi(v')=\Gamma ^- (v).$ Let $\Theta $ be the reflexive
closure of $\Phi$ (obtained from $\Phi$  by adding loops everywhere). Let $F$ be a subset of $V$
with $v\in F$ and $v'\notin \Theta (F).$ Then clearly $F$ is a $v$--Moser's set of $\Theta$.
By Theorem \ref{main}, $$|\G _{\Phi}(F)|=|\G _{\Theta}(F)|\ge |\Theta (v)|-1=r.$$ By Menger's Theorem,
$\Phi$ contains $r$ disjoint paths from $\Gamma (v)$ to $\Gamma ^-(v')=\Gamma ^-(v).$
Adding $v$ to each of these paths, we see the existence of $r$ cycles with the desired property.
\end{proof}

\section{The main result}

 %We start by a reduction method.
  Let $\Gamma =(V,E)$ be a locally-finite  reflexive graph.
 We define the {\em weak connectivity} of $\Gamma$ as
\begin{equation}  \label{eq:kappa}
\kappa _0 (\Gamma)=\min  \{|\partial (X)|\  : \ \ 1\le |X|<\infty \}.
\end{equation}
This concept is not  intersting in the finite case, since  $\kappa _0 (\Gamma)=0,$ for any finite graph $\Gamma.$
%The reader  interested only in the finite case may ignore this section.
 A subset $X$ achieving the  minimum in  (\ref{eq:kappa}) is called a
{\em weak fragment} of $\Gamma.$ A weak fragment with minimum cardinality
 will be called a {\em weak atom}.

\begin{proposition}  {
Let $\Gamma =(V,E)$ be a 
 locally-finite  reflexive vertex-transitive graph and let $A$ be a weak atom. Then the subgraph $\Gamma [A]$ induced on $A$ is a vertex-transitive graph.  Moreover every vertex belongs to some weak atom.
\label{atom}} \end{proposition}

\begin{proof}
 Let  $F_1$ and $F_2$ are weak
 fragments with $F_1\cap F_2 \neq \emptyset.$

By (\ref{submodularity}),
$2\kappa _0 (\Gamma)\le |\partial  (F_1 \cap F_2)|+|\partial  (F_1 \cup F_2)|\le |\partial  (F_1 )|+|\partial ( F_2)|\le 2\kappa _0 (\Gamma).$
Hence $F_1\cap F_2$ is a weak fragment. In particular, distinct weak atoms are disjoint.
Take $v\in A.$ For every $y\in V,$ there is an automorphism $\phi$ such that $\phi (v)=y.$ Hence $y$ belongs to the atom $\phi (A).$  If $y\in A,$ then $\phi(A)=A,$ and therefore $\phi/A$ is an automorphism of $\Gamma[A]$ with $\phi/A(v)=y.$ Thus $\Gamma [A]$  is a vertex-transitive graph.
\end{proof}

We are now ready to prove the general form of the main result:

\begin{theorem}\label{main}
Let $\Gamma =(V,E)$  be a  reflexive locally-finite vertex-transitive 
graph and  let $v\in V.$ Let $F$ be a finite subset of $V$ with $v\in F.$ Then $$|\Gamma (F)|\ge |F|+ |\Gamma  (v)|-|\Gamma ^- (v)\cap F|.$$

\end{theorem}
\begin{proof}

The result holds by Corollary \ref{finitemain}, if $V$ is finite. Suppose that $V$ is infinite. By Proposition \ref{atom}, there is a weak atom
$A$ with $v\in A.$
 By Proposition \ref{atom}, $\Gamma [A]$ is a vertex-transitive graph.  Clearly $F\cap A$ is Moser's set of the finite graph $\Gamma [A].$  By Corollary \ref{finitemain}, $$
| \partial (F\cap A)|\ge |A\cap \Gamma (v)|-|A\cap \Gamma^- (v)|.$$

 By the definition of $\kappa _0,$ we have $|\partial (F\cup A)|\ge \kappa _0=|\partial (A)|.$
Notice that $\Gamma (v)\subset A \cup \G (A)$ and that $\partial (F\cup A)\setminus \partial (F)\subset \partial (A) \setminus \Gamma (v).$ 
It follows that  \begin{eqnarray*}|\partial (F)|&\ge& | \G (F\cap A)\cap A|+ |\partial (F\cap A)\cap \partial (F)|\\ &\ge&
| \G (F)\cap A|+ |\partial (F\cup A)|-|\partial (A) \setminus
 \Gamma (v)|\\&\ge& |A\cap \Gamma (v)|-|A\cap \Gamma ^- (v)|+|\partial (A)|-|\partial (A)\setminus \Gamma (v)|\\&=& |A\cap \Gamma (v)|-|A\cap \Gamma ^- (v)|+|\partial (A)\cap \Gamma (v)|\ge |\Gamma (v)|-|F\cap \Gamma ^- (v)|.\end{eqnarray*}\end{proof}

\begin{corollary}\label{mskw} (Kemperman's Theorem \cite{kempcompl})

Let  $A$ and $B$ be a finite subsets of a group $G$
and let $c\in AB.$ Then $|AB|\ge |A|+|B|-|A\cap (cB^{-1})|.$

\end{corollary}
\begin{proof}
Take $c=ab,$ with $a\in A$ and $b\in B.$ Put $A'=a^{-1}A $ and  $B'=Bb^{-1}.$
Put $\Gamma =\mbox{Cay} (G,B').$ Notice that $\Gamma$ is a reflexive vertex-transitive graph
and that $A'$ is a $1$--Moser's set of $\Gamma.$ By Theorem \ref{main},
\begin{eqnarray*}
|AB|=|A'B'|&=& |\Gamma (A')|\\&\ge& |A'|+|\Gamma (1)|-|A'\cap \Gamma ^{-1}(1)|\\&=&|A'|+|B'|+|A'\cap B'^{-1}| =|A|+|B|-|A\cap (cB^{-1})|.\end{eqnarray*}
\end{proof}

%\end{document}

\end{document}